\documentclass[12pt,letterpaper]{article}
\usepackage[margin=1in]{geometry}

\usepackage{amsmath,amsthm,amssymb}
\usepackage{tikz-cd}
\usepackage{hyperref}
\usepackage[nameinlink,noabbrev,capitalise]{cleveref}
\hypersetup{
  hidelinks,
  pdftitle={A minimising principle for interval covers of zigzag persistence modules}
}

\theoremstyle{plain}
\newtheorem{thm}{Theorem}[section]

\newtheorem{cor}[thm]{Corollary}

\theoremstyle{definition}

\newtheorem{ex}[thm]{Example}
\numberwithin{equation}{section}

\newcommand{\F}{\mathbb{F}}
\newcommand{\id}{\mathrm{id}}
\DeclareMathOperator{\rank}{rank}
\DeclareMathOperator{\Supp}{Supp}
\DeclareMathOperator{\im}{im}

\title{\bfseries A minimising principle for interval covers of zigzag persistence modules}

\newif\ifanonymous
\anonymousfalse
\ifanonymous
\author{}
\else
\author{%
  Junyan Chu \qquad Shizuo Kaji \qquad Meiyan Kang%
}
\fi

\date{}

\begin{document}

\maketitle

\begingroup
\renewcommand{\thefootnote}{}
\makeatletter
\renewcommand{\@makefntext}[1]{\noindent#1}
\makeatother
\begin{NoHyper}
\footnotetext[0]{%
\par
}
\end{NoHyper}
\endgroup

\begin{abstract} 
Zigzag persistence modules serve as an essential tool in topological data analysis.
The fundamental theorem of persistent homology states that any zigzag persistence module decomposes into a direct sum of simple building blocks called interval modules.
We consider an interval cover, a possibly redundant family of
interval submodules spanning every vertex space of a finite zigzag persistence module.
We give an elementary descent argument that replaces a redundant cover with a minimal one that is an interval decomposition.
This provides an optimisation perspective on persistence modules and their interval decomposition.
As an application, we give an estimate of the generalised
rank invariant in terms of a cover.
\end{abstract}

\section{Introduction}\label{sec:intro}

A finite zigzag persistence module assigns finite-dimensional vector spaces
to the vertices of an oriented path and linear maps to its edges.
More precisely,
fix a field $\F$, a positive integer $M$, and an orientation $Q$ of the path
\[
                 1\;-\;2\;-\;\cdots\;-\;M.
\]
A \emph{zigzag persistence module} $V$ over $Q$ consists of
finite-dimensional $\F$-vector spaces $V_1,\ldots,V_M$ and one linear map
$h_e$ along each edge $e$ in its prescribed direction.
By an \emph{interval} we mean a nonempty set
$I=[a,b]=\{a,a+1,\ldots,b\}\subseteq\{1,\ldots,M\}$ of consecutive vertices.
Its \emph{length} is $|I|=b-a+1$.
The \emph{interval module} $\F_I$ is $\F$ at the vertices of $I$ and zero
elsewhere, with the identity on every edge contained in $I$ and the zero map
on all other edges.
The interval decomposition theorem expresses every zigzag persistence module as a direct
sum of interval modules, with uniquely determined supports~\cite{zigzag}.
A persistence module typically arises as persistent homology of data, and this decomposition
allows it to be represented conveniently by a multiset of supports called the \emph{barcode}.

We consider a family $\Sigma$ of interval submodules of $V$ that spans the whole module, which we call an \emph{interval cover}.
The cover may have redundancy in the form of linear relations,
and the most efficient one is identified with an interval decomposition.
Namely, write $n_r(\Sigma)$ for the number of its members
containing vertex $r$, and let $\|\Sigma\|=\sum_r n_r(\Sigma)$ be its total
support length.
Our main result \Cref{thm:descent} shows that for any cover satisfying
$\|\Sigma\|>\dim V:=\sum_{r=1}^{M}\dim_\F V_r$,
we can construct a cover with a smaller total support length.
Every interval cover satisfies $\|\Sigma\|\ge\dim V$,
and equality holds precisely for decompositions.
Thus this descent starting from any $\Sigma$ reaches a decomposition
in at most $\|\Sigma\|-\dim V$ steps.
In \Cref{thm:uniqueness} we prove the uniqueness of the barcode
by showing that
for two arbitrary interval covers
with different supports, at least one of them can be shortened.
In \Cref{thm:local-rank} we give
an estimate of the generalised rank invariant using an interval cover.
\medskip

The decomposition theorem is a special case of Gabriel's classification of
representations of type-$A$ quivers~\cite{gabriel1972} and various proofs are known~\cite{zigzag,ringel2013}.
A more general decomposition theorem into
indecomposables for pointwise finite-dimensional persistence
modules and short proofs of interval and block
decomposition theorems are obtained in~\cite{botnan2020decomposition}.
Our key idea is to characterise the interval decompositions
as the minimisers of the total support length over
possibly redundant interval covers.
We believe
interval covers and their redundancy provide a new perspective on the structure of zigzag persistence modules.

\section{Optimisation of interval covers}\label{sec:covers}

We call
a morphism $s\in\operatorname{Hom}(\F_I,V)$ for some interval $I$
a \emph{generalised interval section} of $V$.
Equivalently, it is a family of vectors $s_r\in V_r$ for $r\in I$
which is carried compatibly along every edge inside $I$ and is carried to
zero across every boundary edge directed out of $I$.
Note that no condition is imposed at a boundary edge directed into $I$.
For an interval $L$ and a subinterval $I\subseteq L$, write $\operatorname{Hom}_L$ for morphisms over the restriction $Q|_L$ and call
$s\in\operatorname{Hom}_L(\F_I,V|_L)$
a generalised interval section on $V|_L$. In this case, no conditions are imposed along the edges leaving $L$.

A monomorphism $s\colon\F_I\to V$ is called an \emph{interval section}.
In this case its image, denoted by $\langle s\rangle$, is a submodule of $V$ isomorphic to $\F_I$.

A finite family $\Sigma=\{(I_j,s_j)\}_{j=1}^{m}$ of interval sections is an
\emph{interval cover} if the induced morphism
\[
             \Phi_\Sigma\colon\bigoplus_{j=1}^{m}\F_{I_j}\longrightarrow V
\]
is surjective.
Its \emph{size} and its \emph{support} are
\[
     \|\Sigma\|=\sum_{j=1}^{m}|I_j|,
     \qquad
     \Supp(\Sigma)=\{I_1,\ldots,I_m\}\ \text{ (a multiset)}.
\]

At a vertex $r$, the vectors
$\{s_{j,r}\mid r\in I_j\}$
span $V_r$, so their number
\[
 n_r(\Sigma)=\#\{j\mid r\in I_j\}
\]
is at least $\dim_\F V_r$, and we have
\begin{equation}\label{eq:redundancy}
 \|\Sigma\|
 =\sum_{r=1}^{M}n_r(\Sigma)
 \ge\sum_{r=1}^{M}\dim_\F V_r=\dim V.
\end{equation}
Equality holds exactly when the section values form a basis at every vertex,
or equivalently when $\Phi_\Sigma$ is an isomorphism, giving an interval decomposition.
We call
$\varepsilon_r(\Sigma)=n_r(\Sigma)-\dim_\F V_r=\dim_\F\ker(\Phi_\Sigma)_r$
the \emph{redundancy} of $\Sigma$ at $r$.
For a set $X$ of vertices we write
$\varepsilon_X(\Sigma)=\sum_{r\in X}\varepsilon_r(\Sigma)$, and
$\varepsilon(\Sigma)=\|\Sigma\|-\dim V$ is the total redundancy.

\begin{thm}\label[thm]{thm:descent}
Let $\Sigma$ be an interval cover of $V$ with $\varepsilon_q(\Sigma)>0$.
Then there is an interval cover $\Sigma'$ with
\[
 n_r(\Sigma')\le n_r(\Sigma)\quad\text{for every vertex }r,
 \qquad
 n_q(\Sigma')\le n_q(\Sigma)-1.
\]
In particular $\|\Sigma'\|\le\|\Sigma\|-1$.
\end{thm}

Any nonzero vector $v\in V_r$ can be propagated along the outgoing
paths, stopping along a
path either when the next edge points back towards $r$ or when the image
becomes zero, to make an interval section. So interval covers exist for any $V$ and
\begin{cor}\label[cor]{cor:decomposition}
Every finite zigzag persistence module $V$ has a decomposition
\[
V\cong\bigoplus_{j=1}^{m}\F_{I_j}.
\]
More precisely,
\[
       \min_{\Sigma\text{ an interval cover of }V}\|\Sigma\|=\dim V,
\]
and an interval cover $\Sigma=\{(I_j,s_j)\}_{j=1}^{m}$ is a minimiser if and
only if it gives an interval decomposition
$V=\bigoplus_{j=1}^{m}\langle s_j\rangle$.
\end{cor}

The proof uses two operations on generalised interval sections.
First, restricting a generalised interval section to
the maximal subintervals where its values are nonzero gives interval sections,
called its \emph{pieces} (since any edge leaving a piece inside $I$ lands at a zero value).
Second, to add two sections, we give them a common domain by
precomposing with morphisms between interval modules, which we now describe.

Let $I=[a,b]$ and $J=[c,d]$.
Commutation with the identity structure maps shows that a morphism
$\F_I\to\F_J$ is multiplication by a single scalar on $I\cap J$ and zero
elsewhere.
If the scalar is nonzero, normalise it to $1$. Commutation at the left
boundary of $I\cap J$ then allows only the squares
\[
\begin{tikzcd}[column sep=large,row sep=small]
a-1 & a \\
0 \arrow[r,"0"] \arrow[d,"0"']
  & \F \arrow[d,"\id"] \\
\F \arrow[r,"\id"'] & \F
\end{tikzcd}
\quad (a>c),
\qquad
\begin{tikzcd}[column sep=large,row sep=small]
c-1 & c \\
\F \arrow[d,"0"']
  & \F \arrow[l,"\id"'] \arrow[d,"\id"] \\
0 & \F \arrow[l,"0"]
\end{tikzcd}
\quad (a<c).
\]
Intuitively, $0$ can only be initial or terminal in a commutative square.
The right boundary is analogous.
Accordingly, we define $J\preceq I$ if
\begin{equation}\label{eq:order}
\begin{aligned}
a>c&\ \Longrightarrow\ a-1\to a,
& a<c&\ \Longrightarrow\ c\to c-1,\\
b<d&\ \Longrightarrow\ b+1\to b,
& b>d&\ \Longrightarrow\ d\to d+1,
\end{aligned}
\end{equation}
that is, wherever the endpoints differ, the edge just outside the inner
endpoint points towards the endpoint of $I$.
When all arrows point right, $[c,d]\preceq[a,b]$ means $c\le a$ and
$d\le b$.
Thus a nonzero morphism $\F_I\to\F_J$ exists exactly when
$I\cap J\ne\emptyset$ and $J\preceq I$, and we write $\phi_{I,J}$ for the
one equal to the identity on $I\cap J$.
In particular $\operatorname{End}(\F_I)=\F$.
The relation $\preceq$ is a partial order, and it is total on intervals with
a common left endpoint or a common right endpoint.
Indeed, the left endpoint conditions in \eqref{eq:order} define a total
order: insert the endpoints $a=1,\ldots,M$ successively, placing each
$a>1$ after all preceding endpoints if $a-1\to a$, and before them otherwise.
Similarly, insert the right endpoints $b=M,\ldots,1$ successively, placing
each $b<M$ before all preceding endpoints if $b\to b+1$, and after them
otherwise.  The relation $\preceq$ is the restriction of the product of
these two total orders to the pairs $(a,b)$ with $a\le b$.

For a subpath $L$ and intervals $I$ and $J$ meeting $L$, write
$J\preceq_L I$ if $J\cap L\preceq I\cap L$.
For intervals through a vertex $q$, the relation $\preceq_{[1,q]}$ compares
only their left endpoints and $\preceq_{[q,M]}$ only their right endpoints;
both are total preorders, and their intersection is $\preceq$.
Let $L$ be $[1,q]$, $[q,M]$, or $[1,M]$.
If $J\preceq_L I$, a generalised section $s\colon\F_{J\cap L}\to V|_L$
transfers to $I\cap L$ as $s\circ\phi_{I\cap L,J\cap L}$, whose values are
$s_r$ on $I\cap J\cap L$ and zero elsewhere; in particular, the value at $q$
is unchanged.
For members $s_i$ and $s_j$ through $q$ with $I_i\preceq_L I_j$, we write
$s_i^{(j)}$ for the transfer of $s_i|_L$ to $I_j\cap L$.

\begin{proof}[Proof of \Cref{thm:descent}]
By scaling, we assume a dependency at $V_q$ is of the form
\begin{align}\label{eq:thm2.2relation}
       s_{1,q}+\cdots+s_{m,q}=0, \quad (m\ge 2)
\end{align}
among distinct sections $s_i$ in $\Sigma$ supported on $I_i=[a_i,b_i]\ni q$.
We argue by induction on $m$.
Below, 
sections in $\Sigma$ outside the relation \eqref{eq:thm2.2relation} are retained, and
every new section below is replaced by its pieces.

\smallskip
\noindent\emph{Case 1: $I_i\preceq I_j$ for some $i\ne j$.}
Replace $s_j$ by $t=s_j+s_i^{(j)}$ on $I_j$.
Since $s_i$ is retained, the spans are unchanged.
If $t_q=s_{i,q}+s_{j,q}$ vanishes, this new
cover $\Sigma'$ has a smaller $n_q$.
Otherwise, the piece of $t$ containing $q$ and the $s_h$ with $h\ne i,j$
satisfy a relation with $m-1$ terms at $q$, and the induction hypothesis
applied to the new cover gives $\Sigma'$.

\smallskip
\noindent\emph{Case 2: otherwise.}
Then $I_1,\ldots,I_m$ are pairwise incomparable under $\preceq$, so
$\preceq_{[1,q]}$ and $\preceq_{[q,M]}$ restrict to opposite strict total
orders on them.  We label them so that
\[
 I_h\prec_{[1,q]}I_j\quad\text{and}\quad I_j\prec_{[q,M]}I_h
 \qquad(h<j),
\]
and, using the transfers on $[1,q]$ and $[q,M]$ respectively, form the cumulative sums
\[
 L_j=\sum_{h\le j}s_h^{(j)}\quad\text{on }[a_j,q],
 \qquad
 R_j=-\sum_{h\ge j}s_h^{(j)}\quad\text{on }[q,b_j].
\]
By the relation \eqref{eq:thm2.2relation},
\[
 L_{j,q}=R_{j+1,q}\quad(1\le j<m),
 \qquad
 L_{m,q}=R_{1,q}=0.
\]
Hence $L_j$ and $R_{j+1}$ glue to a morphism
$G_j\colon\F_{[a_j,b_{j+1}]}\to V$ for $1\le j<m$, and since $L_{m,q}=R_{1,q}=0$, the restrictions of
$L_m$ and $R_1$ to the possibly empty intervals $[a_m,q-1]$ and
$[q+1,b_1]$ are also morphisms.
Replace $s_1,\ldots,s_m$ by $G_1,\ldots,G_{m-1}$, $L_m|_{[a_m,q-1]}$, and
$R_1|_{[q+1,b_1]}$.
At each vertex $r\ne q$ this is an invertible unitriangular change of
$\{s_{j,r}\mid r\in I_j\}$, so the spans are unchanged and no $n_r$ increases, while
only the $m-1$ sections $G_j$ pass through $q$.  This gives $\Sigma'$.
\end{proof}

The construction is illustrated in \Cref{fig:gluing}.
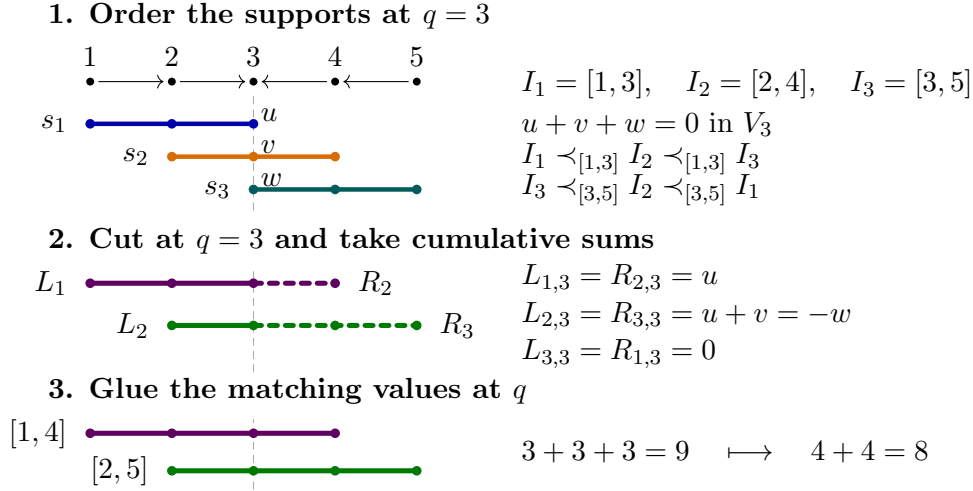
\begin{figure}[htbp]
\centering
\begin{tikzpicture}[
  x=1.08cm,y=.62cm,font=\small,
  sone/.style={blue!68!black,line width=1.7pt,line cap=round},
  stwo/.style={orange!85!black,line width=1.7pt,line cap=round},
  sthree/.style={teal!72!black,line width=1.7pt,line cap=round},
  pairone/.style={violet!75!black,line width=1.9pt,line cap=round},
  pairtwo/.style={green!48!black,line width=1.9pt,line cap=round},
  vertex/.style={circle,fill=black,inner sep=1.05pt},
  qline/.style={gray!65,densely dashed},
  heading/.style={font=\small\bfseries,anchor=west}
]

\node[heading] at (-.65,9.75) {1. Order the supports at $q=3$};
\foreach \x/\r in {0/1,1/2,2/3,3/4,4/5} {
  \node[vertex,label=above:$\r$] at (\x,8.35) {};
}
\draw[->] (.1,8.35)--(.9,8.35);
\draw[->] (1.1,8.35)--(1.9,8.35);
\draw[->] (2.9,8.35)--(2.1,8.35);
\draw[->] (3.9,8.35)--(3.1,8.35);
\draw[qline] (2,5.6)--(2,8.05);

\draw[sone] (0,7.45)--(2,7.45);
\foreach \x in {0,1,2} \fill[blue!68!black] (\x,7.45) circle (1.8pt);
\node[anchor=east] at (-.15,7.45) {$s_1$};
\node[anchor=south west,inner sep=1pt] at (2.05,7.45) {$u$};

\draw[stwo] (1,6.75)--(3,6.75);
\foreach \x in {1,2,3} \fill[orange!85!black] (\x,6.75) circle (1.8pt);
\node[anchor=east] at (.85,6.75) {$s_2$};
\node[anchor=south west,inner sep=1pt] at (2.05,6.75) {$v$};

\draw[sthree] (2,6.05)--(4,6.05);
\foreach \x in {2,3,4} \fill[teal!72!black] (\x,6.05) circle (1.8pt);
\node[anchor=east] at (1.85,6.05) {$s_3$};
\node[anchor=south west,inner sep=1pt] at (2.05,6.05) {$w$};

\node[anchor=west] at (5.15,8.3)
  {$I_1=[1,3],\quad I_2=[2,4],\quad I_3=[3,5]$};
\node[anchor=west] at (5.15,7.45) {$u+v+w=0$ in $V_3$};
\node[anchor=west] at (5.15,6.7)
  {$I_1\prec_{[1,3]}I_2\prec_{[1,3]}I_3$};
\node[anchor=west] at (5.15,6.05)
  {$I_3\prec_{[3,5]}I_2\prec_{[3,5]}I_1$};

\node[heading] at (-.65,4.95) {2. Cut at $q=3$ and take cumulative sums};
\draw[qline] (2,2.15)--(2,4.35);

\draw[pairone] (0,4.05)--(2,4.05);
\draw[pairone,dashed] (2,4.05)--(3,4.05);
\foreach \x in {0,1,2,3} \fill[violet!75!black] (\x,4.05) circle (1.8pt);
\node[anchor=east] at (-.15,4.05) {$L_1$};
\node[anchor=west] at (3.15,4.05) {$R_2$};

\draw[pairtwo] (1,3.15)--(2,3.15);
\draw[pairtwo,dashed] (2,3.15)--(4,3.15);
\foreach \x in {1,2,3,4} \fill[green!48!black] (\x,3.15) circle (1.8pt);
\node[anchor=east] at (.85,3.15) {$L_2$};
\node[anchor=west] at (4.15,3.15) {$R_3$};

\node[anchor=west] at (5.15,4.15)
  {$L_{1,3}=R_{2,3}=u$};
\node[anchor=west] at (5.15,3.35)
  {$L_{2,3}=R_{3,3}=u+v=-w$};
\node[anchor=west] at (5.15,2.55)
  {$L_{3,3}=R_{1,3}=0$};

\node[heading] at (-.65,1.75) {3. Glue the matching values at $q$};
\draw[qline] (2,-.35)--(2,1.15);
\draw[pairone] (0,.85)--(3,.85);
\foreach \x in {0,1,2,3} \fill[violet!75!black] (\x,.85) circle (1.8pt);
\node[anchor=east] at (-.15,.85) {$[1,4]$};
\draw[pairtwo] (1,.05)--(4,.05);
\foreach \x in {1,2,3,4} \fill[green!48!black] (\x,.05) circle (1.8pt);
\node[anchor=east] at (.85,.05) {$[2,5]$};
\node[anchor=west] at (5.15,.45)
  {$3+3+3=9\quad\longmapsto\quad4+4=8$};
\end{tikzpicture}
\caption{
Let $s_1,s_2,s_3$ be interval sections on $I_1=[1,3]$, $I_2=[2,4]$,
and $I_3=[3,5]$ whose values $u,v,w$ at $q=3$ satisfy $u+v+w=0$.
The supports are pairwise incomparable, with
$I_1\prec_{[1,3]}I_2\prec_{[1,3]}I_3$ and
$I_3\prec_{[3,5]}I_2\prec_{[3,5]}I_1$, so Case~2 of the proof of
\Cref{thm:descent} applies with this labelling.
The pairs $L_1,R_2$ and $L_2,R_3$ have matching values at $3$ and glue to
sections on $[1,4]$ and $[2,5]$, while $L_3$ and $R_1$ vanish.
The total size decreases from $9$ to $8$.
}
\label{fig:gluing}
\end{figure}

\begin{ex}\label[ex]{ex:forward-descent}
The argument gives a particularly simple proof for 
an ordinary persistence module $V_1\to\cdots\to V_M$, where
all arrows point right.
Write $h_r\colon V_r\to V_{r+1}$, set $V_0=V_{M+1}=0$ and $h_0=h_M=0$,
and let $h_{q,r}$ denote the composite from $r$ to $q$, with $h_{r,r}=\id$.
For $0\ne v\in V_r$, define its \emph{remaining lifetime} by
\[
 \ell(v)=\min\{1\le k\le M+1-r\mid h_{r+k,r}v=0\},\qquad \ell(0)=0.
\]
The nonzero images of $v$ form an interval section on $[r,r+\ell(v)-1]$, called the \emph{orbit} of $v$; conversely, every interval section is the orbit of its initial value.

Here \Cref{thm:descent} reduces to shortening one orbit, without gluing.
Given a dependence $\sum_i c_i h_{q,r_i}v_i=0$ among nonzero orbit values
at $q$, with all $c_i\ne0$, choose $j$ with largest $r_j$ and replace $v_j$ by
\[
 \bar v_j=\sum_i c_i h_{r_j,r_i}v_i,
 \qquad \ell(\bar v_j)\le q-r_j<\ell(v_j).
\]
The orbit cover still spans with a shorter remaining lifetime: the replacement is an invertible change of generators since $c_j\ne0$.

We can give an interpretation of the elder rule~\cite{cai2021elder}
as a decrease of total support length as well.
For a merger $h_{q,a}u=h_{q,b}v\ne0$ with $a<b$, the above replaces $v$ by
$v-h_{b,a}u$ so the older orbit is retained and the younger one ends before $q$.

We can also identify some splitting components from the minimisation viewpoint.
Fix a nonzero class in $V_r/\im h_{r-1}$ and choose
a representative $v$ of least remaining lifetime.  
We show its orbit $s$ splits off by
proving
\begin{equation}\label{eq:forward-splitting}
 \min_{\Sigma\text{ an interval cover of }V\text{ containing }s}\|\Sigma\|=\dim V.
\end{equation}
Minimise $\|\Sigma\|$ among covers containing $s$.
If there is a dependence, it must involve $s$. Propagate it to $d=r+\ell(v)-1$ and discard zero terms.
If all other surviving terms start before $r$, their combination shortens
$v$ within its fixed class, a contradiction.
Otherwise, choose another term of latest birth (necessarily at least $r$)
and shorten it at $d$.
This descent preserves $s$ and contradicts minimality.

For example, every nonzero vector at the first nonzero vertex ($h_{r-1}=0$), and 
a nonzero vector of maximal remaining lifetime among all vectors
(if an element of $\im h_{r-1}$ shortens the remaining lifetime of $v$,
a preimage of that element in $V_{r-1}$ would have a longer remaining lifetime)
satisfy the condition.

The splitting does not hold for zigzag.
On $1\leftarrow2\to3$, take
$V=\F_{[1,2]}\oplus\F_{[2,3]}$ with the coordinate projections
$\F^2\to\F$.  The values $1,(1,1),1$ form a section on $[1,3]$, which has the largest support.
A complement would be supported at $2$, but the projections have zero
common kernel.
\end{ex}

\section{Uniqueness of interval supports}\label{sec:uniqueness}

We prove the uniqueness of the multiset of interval supports by showing any two covers with different supports must have a redundancy somewhere.

\begin{thm}\label[thm]{thm:uniqueness}
Let $\Sigma$ and $\mathcal T$ be interval covers of $V$ with
$\Supp(\Sigma)\ne\Supp(\mathcal T)$.
Then at least one of them can be replaced by a cover of strictly smaller size.
Consequently, all minimal interval covers of $V$ have the same multiset of
supports.
\end{thm}

\begin{proof}
Let $\Sigma=\{(I_j,s_j)\}$ and $\mathcal T=\{(J_k,t_k)\}$ be interval covers
with their respective multiplicities $m_K,n_K$ of an interval $K$.
Choose a $\preceq$-minimal interval $I$ with $m_I\ne n_I$. We can assume $m_I>n_I$.
If $\mathcal T$ has redundancy, \Cref{thm:descent} shortens $\mathcal T$.
So we assume $\mathcal T$ has no redundancy.
For any $q\in I$, consider two families
\[
 \{s_{j,q}\mid q\in I_j,\ I_j\preceq I\}, \quad
 \{t_{k,q}\mid q\in J_k,\ J_k\preceq I\}.
\]
Since $\Phi_{\mathcal T}$ is an isomorphism, each $s_j\in\Sigma$ gives a morphism
$\Phi_{\mathcal T}^{-1}\circ s_j\colon\F_{I_j}\to \bigoplus_k\F_{J_k}$,
and its $k$th component is nonzero
only if $J_k\preceq I_j$.
By transitivity of $\preceq$, the span of the first family is contained in the span of the second.
By the choice of $I$, the sizes of the above families differ by
\[
 \sum_{\substack{K\preceq I\\q\in K}}(m_K-n_K)=m_I-n_I>0.
\]
Since the second family is linearly independent, the first family
contains more vectors than the dimension of the span of the second,
making it dependent.
\end{proof}

The common multiset of supports, denoted by $\mathcal B(V)$, is called the \emph{barcode} of $V$.

\section{Estimating the rank invariant from interval covers}\label{sec:rank-bound}

We express the generalised rank invariant~\cite{kim2021generalized} through
interval covers, and bound it using the redundancy of a cover.
For an interval $J$, set
\[
 \Gamma_J(V)=\operatorname{Hom}_J(\F_J,V|_J),
 \qquad
 C_J(V)=\operatorname{Hom}_J(V|_J,\F_J)^*.
\]
Here $(-)^*$ denotes the vector-space dual.
Since $\operatorname{End}_J(\F_J)=\F$, composition defines a natural map
\[
 \lambda_J^V\colon\Gamma_J(V)\longrightarrow C_J(V),
 \qquad \lambda_J^V(s)(t)=t\circ s.
\]
Its rank $\rho_J(V)=\rank\lambda_J^V$ is called
the \emph{generalised rank invariant}~\cite{kim2021generalized}.
The spaces $\Gamma_J(V)$ and $C_J(V)$ identify naturally with the limit
and colimit of $V|_J$, respectively, and $\lambda_J^V$ is the usual
limit-to-colimit map.
In particular, when all arrows point right, $\rho_{[a,b]}(V)=\rank h_{b,a}$.

For an interval module $\F_I$ with $I\cap J\ne\emptyset$, we have
$\F_I|_J=\F_{I\cap J}$, so 
$\Gamma_J(\F_I)\ne0$ exactly when $I\preceq_J J$, and $C_J(\F_I)\ne0$
exactly when $J\preceq_J I$.
Hence $\rho_J(\F_I)=1$ if $J\subseteq I$ and zero otherwise.
Additivity gives
$\rho_J(\bigoplus_{j=1}^{m}\F_{I_j})=\#\{j\mid J\subseteq I_j\}$.

For a cover $\Sigma=\{(I_j,s_j)\}$, define the corresponding count
\[
 p_J(\Sigma)=\#\{j\mid J\subseteq I_j\}.
\]
In particular, $p_J(\Sigma)=\rho_J(V)$ for
any interval decomposition.
A redundant cover may under- or over-count.
We bound this error by the total redundancy:
\begin{equation}\label{eq:rank_error}
 |p_J(\Sigma)-\rho_J(V)|\le\varepsilon(\Sigma),
    \end{equation}
with sharper bounds involving only the sinks of $J$.

For $J=[a,b]$, let $S_J$ be the sinks of $Q|_J$, namely the vertices with
no outgoing edge in $Q|_J$, and let $S^\circ_J=S_J\setminus\{a,b\}$ be the
internal sinks.

\begin{thm}\label[thm]{thm:local-rank}
For every interval cover $\Sigma$ of $V$ and every interval $J$,
\begin{equation}\label{eq:local-rank}
 p_J(\Sigma)-\varepsilon_{S_J}(\Sigma)\le\rho_J(V)\le
 p_J(\Sigma)+\varepsilon_{S^\circ_J}(\Sigma).
\end{equation}
\end{thm}

In particular, if there is no redundancy at any sink of $J$, then
$p_J(\Sigma)=\rho_J(V)$ even if the cover is redundant elsewhere.
If $J$ has no internal sink, the upper bound reduces to
$\rho_J(V)\le p_J(\Sigma)$; this includes intervals whose arrows all point
in the same direction.
Since $\varepsilon_{S^\circ_J}(\Sigma)\le\varepsilon_{S_J}(\Sigma)
\le\varepsilon(\Sigma)$, the global error bound \eqref{eq:rank_error}
follows.

\begin{proof}
Let $A=\bigoplus_j\F_{I_j}$, $\Phi=\Phi_\Sigma\colon A\to V$, and $K=\ker\Phi$, so that we have an exact sequence
\[
\begin{tikzcd}
0 \arrow[r] & K \arrow[r, "\iota"] & A \arrow[r, "\Phi"] & V \arrow[r] & 0
\end{tikzcd}
\]
and $\dim K_x=\varepsilon_x(\Sigma)$.
We have a commutative diagram with exact rows:
\[
\begin{tikzcd}[column sep=large,row sep=large]
 0 \arrow[r]
 & \Gamma_J(K) \arrow[r, "\iota^\Gamma_*"] 
 \arrow[d,"\lambda_J^K"']
 & \Gamma_J(A) \arrow[r,"\Phi^\Gamma_*"] \arrow[d,"\lambda_J^A"']
 & \Gamma_J(V) \arrow[d,"\lambda_J^V"] & \\
 & C_J(K) \arrow[r, "\iota^C_*"]
 & C_J(A) \arrow[r, "\Phi^C_*"]
 & C_J(V) \arrow[r] & 0.
\end{tikzcd}
\]
Commutativity  gives $\Phi^C_*\circ\lambda_J^A
=\lambda_J^V\circ\Phi^\Gamma_*.$
Moreover, exactness at $C_J(A)$ implies
$\ker\Phi^C_*=\operatorname{im}\iota^C_*,$
and hence
$\dim\ker\Phi^C_*\leq\dim C_J(K).$
Note that $\rank\lambda_J^A=p_J(\Sigma)$ and
 $\rank\lambda_J^V=\rho_J(V)$.
It follows that
\begin{align*}
p_J(\Sigma)
&\leq
\operatorname{rank}(\Phi^C_*\circ\lambda_J^A)
+\dim\ker\bigl(\Phi^C_*\bigr)\\
&=
\operatorname{rank}(\lambda_J^V\circ\Phi^\Gamma_*)
+\dim\ker\bigl(\Phi^C_*\bigr)\\
&\leq
\rho_J(V)+\dim C_J(K).
\end{align*}
A morphism $K|_J\to\F_J$ is determined by its components at the sinks,
since every vertex has a directed path to a sink.  Thus
\[
 \dim C_J(K)\le\sum_{x\in S_J}\dim K_x
 =\varepsilon_{S_J}(\Sigma).
\]
This yields the lower bound.

For the upper bound,
choose a complementary subspace $U$ such that
$\Gamma_J(V)=\operatorname{im}\bigl(
\Phi^\Gamma_*\bigr)\oplus U.$
By commutativity,
$\lambda_J^V(\operatorname{im}\bigl(
\Phi^\Gamma_*\bigr))
=
\operatorname{im}(\Phi^C_*\circ\lambda_J^A).$
Consequently,
\begin{align*}
\rho_J(V)
&\leq \dim\lambda_J^V(\operatorname{im}\bigl(
\Phi^\Gamma_*\bigr))+\dim\lambda_J^V(U)\\
&\leq \operatorname{rank}(\Phi^C_*\circ\lambda_J^A)
+\dim U\\
&\leq p_J(\Sigma)
+\dim\operatorname{coker}\Phi^\Gamma_*.
\end{align*}
To bound $\dim\operatorname{coker}\Phi^\Gamma_*$, 
we account for the freedom in choosing lifts at the sources.
Let $R_J$ be the set of sources of $Q|_J$, and set
$D_J=\bigoplus_{x\in S^\circ_J}K_x$.
Choose a linear right inverse $\sigma_r\colon V_r\to A_r$ of $\Phi_r$
for each $r\in R_J$.
For $s\in\Gamma_J(V)$, lift $s_r$ by $\sigma_r$ at every source and
propagate these values along directed paths, stopping at the sinks.
Every vertex that is neither a source nor a sink has exactly one incoming
edge, so the only compatibility conditions arise at sinks.
At each internal sink $x$, the two incoming paths give values
$a_{x,1}(s),a_{x,2}(s)\in A_x$ with the same image $s_x$; define
$\delta_\sigma(s)_x=a_{x,1}(s)-a_{x,2}(s)\in K_x$.
This gives a linear map $\delta_\sigma\colon\Gamma_J(V)\to D_J$.
If $\delta_\sigma(s)=0$, these propagated values agree at every internal
sink and therefore define a section of $A|_J$ lifting $s$.
Thus $\ker\delta_\sigma\subseteq\operatorname{im}\Phi^\Gamma_*$ and
\begin{align*}
     \dim\operatorname{coker}\Phi^\Gamma_*
 &=\dim\bigl(\Gamma_J(V)/\operatorname{im}\Phi^\Gamma_*\bigr)
 \leq\dim\bigl(\Gamma_J(V)/\ker\delta_\sigma\bigr) \\
 &=\rank\delta_\sigma\leq\dim D_J=
 \sum_{x\in S^\circ_J}\dim K_x
 =\varepsilon_{S^\circ_J}(\Sigma).
\end{align*}
This yields the upper bound.
\end{proof}

\begin{ex}\label[ex]{ex:rank-sharp}
Consider $1\to2\leftarrow3$ and $V=\F_{[1,3]}$.
The natural interval sections on $[1,2]$ and $[2,3]$ form a cover $\Sigma$.
For $J=[1,3]$, $S^\circ_J=S_J=\{2\}$ and
\[
 p_J(\Sigma)=0,\qquad
 \varepsilon_{S^\circ_J}(\Sigma)=\varepsilon_2(\Sigma)=\varepsilon(\Sigma)=1,
 \qquad \rho_J(V)=1
\]
giving an equality to the upper bound.

For the lower bound, let
$V=\F_{[1,3]}\oplus\F_{[1,1]}\oplus\F_{[3,3]}$.
The two sections on $J=[1,3]$ with vertex values
\[
 ((1,0),1,(1,0))\qquad\text{and}\qquad ((1,1),1,(1,1))
\]
form a cover $\Sigma$.  We have
\[
 p_J(\Sigma)=2,\qquad \varepsilon_{S_J}(\Sigma)=1,
 \qquad \rho_J(V)=1.
\]
 
This cover has two members, while $V$ has three bars.
For $J=[1,2]$ the same values of $p_J$, $\varepsilon_{S_J}$, and $\rho_J$
hold, with all redundancy at the endpoint sink $2$ and $S^\circ_J=\emptyset$.
Hence endpoint sinks cannot be omitted from $S_J$.
\end{ex}

\ifanonymous\else
\section*{Acknowledgements}
The authors are grateful to
Y.~Asao, E.~Escolar, Y.~Hiraoka, Y.~Ike, I.~Obayashi, and H.~Ochiai for their
valuable comments.
This research was partially supported by JST Moonshot R\&D Grant Number JPMJMS2021 and
KAKENHI, Grant-in-Aid for Scientific Research (B) 25K00921.
\fi

\bibliographystyle{plain}
\bibliography{reference}

\ifanonymous\else
\par\bigskip
\noindent\begin{minipage}{\textwidth}
\small
\raggedright
Graduate School of Science, Kyoto University,
Kyoto, 606-8502, Japan\par
\medskip
\begin{tabular}{@{}ll@{}}
Junyan Chu & \href{mailto:chujy626@gmail.com}{\texttt{chujy626@gmail.com}}\\
Shizuo Kaji & \href{mailto:kaji.shizuo.7r@kyoto-u.ac.jp}{\texttt{kaji.shizuo.7r@kyoto-u.ac.jp}}\\
Meiyan Kang & \href{mailto:kang.meiyan.2c@kyoto-u.ac.jp}{\texttt{kang.meiyan.2c@kyoto-u.ac.jp}}
\end{tabular}
\end{minipage}\par
\fi
\end{document}